\documentclass[letterpaper, 10 pt, conference]{ieeeconf}
\IEEEoverridecommandlockouts
\usepackage{lipsum}
\usepackage{amsfonts}
\usepackage{graphicx}
\usepackage{epstopdf}
\usepackage{standalone}
\usepackage{algcompatible}

\ifpdf
  \DeclareGraphicsExtensions{.eps,.pdf,.png,.jpg}
\else
  \DeclareGraphicsExtensions{.eps}
\fi

\usepackage{amsopn}

\usepackage{epsfig} %
\usepackage{amsmath} %
\usepackage{amssymb}  %
\usepackage{bm}
\DeclareMathAlphabet{\mathcal}{OMS}{cmsy}{m}{n}
\usepackage[english]{babel}
\usepackage{url}
\usepackage{algorithm}
\usepackage{color}
\usepackage{pgfplots}
\usepackage{siunitx}
\usepackage{tikz}
\usepackage{tkz-euclide}
\usepackage{chngcntr}
\usepackage{verbatim}
\usepackage{enumerate}
\usepackage[normalem]{ulem}

\definecolor{ao(english)}{rgb}{0.0, 0.5, 0.0}
\usepackage[colorlinks, citecolor = {ao(english)}, linkcolor = {ao(english)}]{hyperref} 
\usepackage{cleveref}
\usepackage{aliascnt}
\usetikzlibrary{calc}
\usepackage{subcaption}
\usepackage{tikz}
\usepackage{pgfplots}
\usetikzlibrary{arrows,shapes,trees,calc,positioning,patterns,decorations.pathmorphing,decorations.markings}
\usetikzlibrary{matrix}
\usepgfplotslibrary{groupplots}
\pgfplotsset{compat=newest}

\usepackage{amsthm}
\usepackage{amsthm}
\usepgfplotslibrary{patchplots}

\crefname{figure}{Fig.}{Fig.}

\newtheorem{thm}{Theorem}
\crefname{thm}{Theorem}{Theorems}

\newtheorem{prop}{Proposition}
\crefname{prop}{Proposition}{Propositions}

\newtheorem{lem}{Lemma}
\crefname{lem}{Lemma}{Lemmas}
\newtheorem{cor}{Corollary}
\crefname{cor}{Corollary}{Corollaries}
\theoremstyle{remark}

\crefname{rem}{Remark}{Remarks}

\theoremstyle{definition}
\newtheorem{ex}{Example}
\crefname{ex}{Example}{Examples}

\crefname{ass}{Assumption}{Assumption}
\usepackage{dsfont}
\let\mathbb=\mathds

\crefname{conj}{Conjecture}{Conjectures}

\theoremstyle{definition}
\newtheorem{defn}{Definition}
\crefname{defn}{Definition}{Definitions}

\crefname{prob}{Problem}{Problems}
\crefname{algorithm}{Algorithm}{Algorithms}

\crefalias{propenumi}{proposition} 
\newcommand{\nset}[1]{(1:#1)}
\newcommand{\submatrix}[3]{{#1}_{#2,#3}}

\newcommand{\Rmn}{\mathbb{R}^{n \times m}}

\newcommand{\Rnn}{\mathbb{R}^{n \times n}}

\newcommand{\rk}{\textnormal{rank}}

\newcommand{\sign}{\textnormal{sign}}

\newcommand{\transp}{\mathsf{T}}
\newcommand{\vari}[1]{\text{S}^{-}(#1)}

\newcommand{\pos}[2]{\text{pos}_{#1}(#2)}
\newcommand{\variz}[1]{\text{S}^{+}(#1)}

\newcommand{\Toep}[1]{\mathcal{T}_{#1}}
\newcommand{\Hank}[1]{\mathcal{H}_{#1}}
\newcommand{\Con}[1]{{\mathcal{C}^{#1}}}
\newcommand{\Obs}[1]{{\mathcal{O}^{#1}}}

\newcommand{\compound}[2]{#1_{[#2]}}

\colorlet{FigColor1}{blue}
\colorlet{FigColor2}{red}
\colorlet{FigColor3}{ao(english)}
\colorlet{FigColor4}{orange}
\pgfplotsset{every axis plot/.append style={line width=1.5pt}}

\crefformat{equation}{\textup{#2(#1)#3}}
\crefrangeformat{equation}{\textup{#3(#1)#4--#5(#2)#6}}
\crefmultiformat{equation}{\textup{#2(#1)#3}}{ and \textup{#2(#1)#3}}
{, \textup{#2(#1)#3}}{, and \textup{#2(#1)#3}}
\crefrangemultiformat{equation}{\textup{#3(#1)#4--#5(#2)#6}}%
{ and \textup{#3(#1)#4--#5(#2)#6}}{, \textup{#3(#1)#4--#5(#2)#6}}{, and \textup{#3(#1)#4--#5(#2)#6}}

\Crefformat{equation}{#2Equation~\textup{(#1)}#3}
\Crefrangeformat{equation}{Equations~\textup{#3(#1)#4--#5(#2)#6}}
\Crefmultiformat{equation}{Equations~\textup{#2(#1)#3}}{ and \textup{#2(#1)#3}}
{, \textup{#2(#1)#3}}{, and \textup{#2(#1)#3}}
\Crefrangemultiformat{equation}{Equations~\textup{#3(#1)#4--#5(#2)#6}}%
{ and \textup{#3(#1)#4--#5(#2)#6}}{, \textup{#3(#1)#4--#5(#2)#6}}{, and \textup{#3(#1)#4--#5(#2)#6}}

\crefdefaultlabelformat{#2\textup{#1}#3}

\title{\LARGE \bf Characterizing variation bounding in discrete-time Hankel operators}
\author{Zijian Liu, Tom Ashani and Christian Grussler
	\thanks{The project was supported by the Israel Science Foundation (grant no. 2406/22) and the Bernard M. Gordon Center for Systems Engineering at the Technion -- IIT, while the third author was also a Jane and Larry Sherman Fellow.}
	\thanks{Z. Liu and C. Grussler are with the Stephen B. Klein Faculty of Aerospace Engineering, Technion -- Israel Institute of Technology, 3200003 Haifa, Israel.
	Emails: {\tt\small zijian-liu@campus.technion.ac.il, \;  cgrussler@technion.ac.il}.}
    \thanks{T. Ashani is with the Faculty of Mathematics, Technion -- Israel Institute of Technology, 3200003 Haifa, Israel. Email: {\tt\small tom.ashani@campus.technion.ac.il}}
}

\usepackage[backend=biber, style=ieee, eprint=true,doi=false,url=false,giveninits=true,style=numeric-comp,sorting=none,natbib=true]{biblatex}
\DeclareSourcemap{
	\maps[datatype=bibtex,overwrite=true]{
		\map{
			\step[fieldsource=journal,
			match=\regexp{IEEE\sTransactions\son\sAutomatic\sControl},
			replace={IEEE Trans. Autom. Control}]
			\step[fieldsource=journal,
			match=\regexp{Optimization\sLetters},
			replace={Optim. Lett.}]
			\step[fieldsource=journal,
			match=\regexp{Linear\sAlgebra\sand\sits\sApplications},
			replace={Linear Algebra Appl.}]	
			\step[fieldsource=booktitle,
			match=\regexp{Advances\sin\sNeural\sInformation\sProcessing\sSystems},
			replace={Adv. Neural. Inf. Process. Syst.}]
			\step[fieldsource=booktitle,
			match=\regexp{IEEE\sInternational\sConference\son\sSystems,\sMan\sand\sCybernetics}, 
			replace={IEEE. Int. Conf. Syst. Man. Cybern.}]
			\step[fieldsource=journal,
			match=\regexp{Foundations\sand\sTrends\sin\sMachine\sLearning},
			replace={Found. Trends Mach. Learn.}]
			\step[fieldsource=journal,
			match=\regexp{Foundations\sof\sComputational\sMathematics},
			replace={Found. Comut. Math.}]
			\step[fieldsource=journal,
			match=\regexp{IEEE\sSignal\sProcessing\sMagazine},
			replace={IEEE Signal Process. Mag.}]
			\step[fieldsource=journal,
			match=\regexp{IEEE\sTransactions\son\sControl\sof\sNetwork\s Systems},
			replace={IEEE Trans. Control. Netw. Syst.}]
			\step[fieldsource=booktitle,
			match=\regexp{European\sControl\sConference},
			replace={Eur. Control Conf.}]
			\step[fieldsource=journal,
			match=\regexp{IEEE\sControl\sSystems\sLetters},
			replace={IEEE Control Syst. Lett.}]
			\step[fieldsource=journal,
			match=\regexp{Numerical\sFunctional\sAnalysis\sand\sOptimization},
			replace={Numer. Funct. Anal. Optim.}]
			\step[fieldsource=journal,
			match=\regexp{IEEE\sTransactions\son\sPattern\sAnalysis\sand\sMachine\sIntelligence},
			replace={IEEE Trans. Pattern. Anal. Mach. Intell.}]
			\step[fieldsource=journal,
			match=\regexp{Journal\sof\sGeometry\sand\sPhysics},
			replace={J. Geom. Phys.}]
			\step[fieldsource=booktitle,
			match=\regexp{Proceedings\sof\sthe},
			replace={Proc.}]
			\step[fieldsource=series,
			match=\regexp{Proceedings\sof\sMachine\sLearning\sResearch},
			replace={Proc. Mach. Learn. Res.}]
			\step[fieldsource=booktitle,
			match=\regexp{Conference\son\sDecision\sand\sControl},
			replace={Conf. Decis. Control}]
			\step[fieldsource=booktitle,
			match=\regexp{Conference\son\sLearning\sTheory},
			replace={Conf. Learn. Theory}]
			\step[fieldsource=journal,
			match=\regexp{SIAM\sReview},
			replace={SIAM Rev.}]
			\step[fieldsource=journal,
			match=\regexp{SIAM\sJournal\son\sMatrix\sAnalysis\sand\sApplications},
			replace={SIAM J. Matrix Anal. Appl.}]
			\step[fieldsource=journal,
			match=\regexp{European\sJournal\sof\sControl},
			replace={Eur. J. Control}]	
			\step[fieldsource=series,
			match=\regexp{Cambridge\sTracts\sin\sMathematics},
			replace={Camb. Tracts Math.}]
			\step[fieldsource=collection,
			match=\regexp{Cambridge\sTracts\sin\sMathematics},
			replace={Camb. Tracts Math.}]
			\step[fieldsource=journal,
			match=\regexp{Journal\sof\sGlobal\sOptimization},
			replace={J. Glob. Optim.}]	
			\step[fieldsource=booktitle,
			match=\regexp{American\sControl\sConference},
			replace={Am. Control. Conf.}]
			\step[fieldsource=booktitle,
			match=\regexp{International\sConference\son\sMachine\sLearning},
			replace={Int. Conf. Mach. Learn.}]
			\step[fieldsource=booktitle,
			match=\regexp{IEEE\sInternational\sConference\son\sSoftware\sQuality\sReliability\sand\sSecurity},
			replace={IEEE Int. Conf. Softw. Qual. Reliab. Secur.}]
			\step[fieldsource=booktitle,
			match=\regexp{British\sMachine\sVision\sConference},
			replace={Br. Mach. Vis. Conf.}]
			\step[fieldsource=journal,
			match=\regexp{Physics\sof\sFluids},
			replace={Phys. Fluids}]
			\step[fieldsource=journal,
			match=\regexp{IEEE\sControl\sSystems\sMagazine},
			replace={IEEE Control Syst. Mag.}]
			\step[fieldsource=booktitle,
			match=\regexp{Conference\son\sEmpirical\sMethods\sin\sNatural\sLanguag\sProcessing},
			replace={Conf. Empir. Methods. Nat. Lang. Process.}]
			\step[fieldsource=journal,
			match=\regexp{SIAM\sJournal\son\sOptimization},
			replace={SIAM J. Optim.}]
			\step[fieldsource=journal,
			match=\regexp{Journal\sof\sMachine\sLearning\sResearch},
			replace={J. Mach. Learn. Res.}]
			\step[fieldsource=journal,
			match=\regexp{Journal\sof\sDynamics\sand\sDifferential\sEquations},
			replace={J. Dyn. Differ. Equ.}]
			\step[fieldsource=journal,
			match=\regexp{SIAM\sJournal\son\sControl\sand\sOptimization},
			replace={SIAM J. Control. Optim.}]
			\step[fieldsource=journal,
			match=\regexp{Statistical\sScience},
			replace={Stat. Sci.}]
			\step[fieldsource=journal,
			match=\regexp{Journal\sof\sthe\sRoyal\sStatistical\sSociety},
			replace={J. R. Stat. Soc.}]
			\step[fieldsource=journal,
			match=\regexp{Physica\sD:\sNonlinear\sPhenomena},
			replace={Physica D.}]
			\step[fieldsource=journal,
			match=\regexp{Journal\sof\sthe\sFranklin\sInstitute},
			replace={J. Franklin Inst.}]
			\step[fieldsource=journal,
			match=\regexp{Communications\sin\sInformation\sand\sSystems},
			replace={Commun. Inf. Syst.}]
			\step[fieldsource=journal,
			match=\regexp{Journal\sof\sFluid\sMechanics},
			replace={J. Fluid Mech.}]
			\step[fieldsource=journal,
			match=\regexp{Journal\sof\sDifferential\sEquations},
			replace={J. Differ. Equ.}]
			\step[fieldsource=journal,
			match=\regexp{Mathematische\sZeitschrift},
			replace={Math. Z.}]
            \step[fieldsource=journal,
			match=\regexp{Journal\sof\sOptimization\sTheory\sand\sApplications},
			replace={J. Optim. Theory Appl.}]
\step[fieldsource=booktitle,match=\regexp{International\sConference\son\sLearning\sRepresentations},
			replace={Int. Conf. Learn. Represent}]
		}
	}
}

\AtEveryBibitem{\clearfield{issn}}
\AtEveryBibitem{\clearfield{eprintclass}}
\begin{document}

\maketitle
\thispagestyle{empty}
\pagestyle{empty}

\begin{abstract}
We investigate the $k$-variation bounding property of the discrete-time Hankel operator, i.e., its invariance under the set of signals with a variation (number of sign changes) of at most $k$. Building on existing sign-consistency criteria, it is shown that this property is equivalent to the external positivity of $k+1$ explicitly realized linear discrete-time systems. Thus, making the property tractable via numerical and analytic certificates. We also derive dominant-pole restrictions for the case of $k=1$. A three-node thermal example illustrates the results and distinguishes variation bounding from variation diminishing. 
\end{abstract}

\section{INTRODUCTION}
Linear time-invariant systems with positivity properties provide a foundation for simplified analysis and controller design. Examples include positive realness and dissipativity \cite{khalil2002nonlinear}, which facilitate stability analysis of interconnected networks \cite{arcak2007passivity,yue2025passivity}, as well as \emph{positive systems}, whose inputs, states, or outputs remain nonnegative. Positive systems enable scalable linear-programming methods for Lyapunov-based design \cite{rantzer2018tutorial} and play a central role in nonlinear analysis tools such as Zames--Falb multipliers \cite{turner2021discrete}, monotone systems \cite{hirsch2006monotone,angeli2003monotone}, and differential positivity \cite{forni2016differentially,mostajeran2018positivity}. Related specializations include relaxation systems \cite{willems1976realization,pates2019optimal}, totally positive differential systems \cite{schwarz1970totally}, and cyclic monotone systems \cite{mallet1990poincare}.

Recently, these classes have been revisited through the lens of \emph{variation diminishment}, i.e., the reduction of sign changes by system operators. For example, relaxation systems correspond to variation-diminishing Hankel operators \cite{grussler2020variation}, and variation has been proposed as a Lyapunov-like measure \cite{margaliot2018revisiting}. Building on these ideas, \cite{grussler2020variation,grussler2021internally} characterized and analyzed systems whose Hankel operators diminish the variation of signals with at most \(k\) sign changes by introducing the concept of compound systems, forming a hierarchy between positive systems (\(k=0\)) and relaxation systems (\(k=\infty\)). At the opposite end of this line of research lies the property of \emph{\(k\)-variation bounding}, which confines states or signals to at most \(k\) sign changes. This property has been used to extend autonomous (i.e., unforced) positive and monotone system analysis \cite{weiss2021cooperative,weiss2019generalization2,alseidi2021discrete,katz2025instability,wu2022diagonal}, to develop signal-based approaches for cyclic monotone systems \cite{tong2026unimodal}, and to provide failure guarantees in sparse optimal control \cite{marmary2025tractabledownfallbasispursuit}.

The main objective of this work is to complement these studies by addressing the open challenge of providing a tractable and analyzable characterization of systems with \(k\)-variation bounding Hankel operator. We envision that this provides grounds for an input–output perspectives on relaxed positive system, which will be amenable for capturing overshoot or non-minimum-phase behavior, as well as motivate extensions of nonlinear tools such as Zames--Falb multipliers. Concretely, in this work, we take initial steps towards the characterization of \(k\)-variation bounding Hankel operators for finite-dimensional discrete-time linear time-invariant (DTLTI) systems. Our approach builds on a recent characterization of \((k+1)\)-sign-consistent Hankel operators \cite{grussler2026efficient} and its equivalence to $k$-variation bounding \cite{grussler2024system}: the infinite-dimensional Hankel matrix is \((k+1)\)-sign-consistent when all its minors of order \(k+1\) share the same sign. Applying sign-consistency results from \cite{pena_matrices_1995,grussler2024system}, we show that certifying \(k\)-variation bounding is equivalent to verifying input–output positivity of \(k+1\) associated DTLTI \emph{compound systems}. The input–output positivity of these compound systems can be checked efficiently using existing certificates \cite{grussler2019tractable,taghavian2023external,drummond2023externally,weller2023external}. Our explicit \emph{generalized compound system} realizations also provide the base for further analytical treatment, e.g., for the analysis of pole structures as exemplified here for the case of $k=2$. This is in line with prior work on variation-diminishing Hankel operators \cite{grussler2020variation} and on variation-bounding observability operators \cite{grussler2024system}. Notably, the pole constraint we obtain does not mirror the eigenvalue constraints for autonomous variation-bounding systems \cite{alseidi2021discrete,wu2022diagonal}. Our findings are illustrated by a three-node thermal chain.

The remainder of the manuscript is organized as follows. Preliminaries appear in \Cref{sec:prelim}. Our main results are presented in \Cref{sec:main}. A thermal example is presented in \Cref{sec:ex}, and we conclude in \Cref{sec:conc}.

\section{PRELIMINARIES}
\label{sec:prelim}
\subsection{Notations}
We write $\mathds{Z}$ for the set of integers and $\mathds{R}$ for the set of reals, with  $\mathds{Z}_{\ge 0}$ and $\mathds{R}_{\ge 0}$ standing for the respective subsets of nonnegative elements -- the corresponding notations are also used for subsets starting from non-zero values, strict inequality as well as reversed inequality signs. The set of real sequences with indices in $\mathbb{Z}$ is denoted by 
$\mathbb{R}^{\mathbb{Z}}$. For matrices $X = (x_{ij}) \in \Rmn$, we say that $X$ is
\emph{nonnegative},
$X \geq 0$ 
or $X \in \Rmn_{\geq 0}$
if all elements $x_{ij} \in \mathbb{R}_{\geq 0}$ -- corresponding notations are used for matrices with strictly positive entries and reversed inequality signs. These notations are also used for sequences $x = (x_i) \in \mathbb{R}^{\mathbb{Z}}$. For $k, l \in \mathds{Z}$, we write $(k:l) := \{k,k+1,\dots,l\}$, $k \leq l$. In the case $k > l$, the notation represents the empty set. If $X\in \Rnn$, then $\sigma(X) = \{\lambda_1(X),\dots,\lambda_n(X)\}$ denotes its \emph{spectrum}, where the eigenvalues are ordered by descending absolute value, i.e., $\lambda_1(X)$ is the eigenvalue with the largest magnitude, counting multiplicity. If the magnitude of two eigenvalues coincides, we sub-sort them by decreasing real part. %

For $X \in \Rmn$, the submatrix with rows $I \subset
(1:n)$ and columns $J \subset (1:m)$ is written as $\submatrix{X}{I}{J}$. In the case of subvectors, we simply write $x_I$. With slight abuse of notation, we also use this to denote subsets of ordered index sets $x \in \mathcal{I}_{n,r}$, where 
\begin{equation*}
\mathcal{I}_{n,r} := \{ v = \{v_1,\dots,v_r\} \subset \mathds{N}: 1\leq v_1 < \dots < v_r \leq n \}.
\end{equation*} 
The $i$-th element of $\mathcal{I}_{n,r}$ is taken with respect to the lexicographic ordering on $\mathcal{I}_{n,r}$. Moreover, if $i \in \mathds{N}$ is an element of $v \in \mathcal{I}_{n,r}$, we denote its \emph{position within $v$} by $\pos{v}{i}$, i.e., $\pos{v}{i} = l$ if $i = v_l$. 

The \emph{determinant} of $X \in \Rnn$ is denoted by $|X|$, which can be computed using the \emph{generalized Laplace expansion along the columns with index $I \in \mathcal{I}_{n,r}$} (see, e.g., \cite[Sec.~0.8.9]{horn2012matrix}) as 
\begin{equation}
    |X| = \sum_{\substack{J\in\mathcal{I}_{n,r}}}(-1)^{\sum_{j\in J}j+\sum_{i\in I}i}|X_{J,I}||X_{(1:n)\setminus J,(1:n) \setminus I}|.
    \label{eq:gen_laplace}
\end{equation}
A \emph{(consecutive) $j$-minor} of $X \in \Rmn$ is a minor which is constructed of $j$ columns and $j$ rows of $X$ (with consecutive indices). 

\subsection{Variation diminishing maps}
\label{sec:vardim}
The \emph{variation} of a sequence or vector $u$ is defined 
as the number of sign changes in $u$. We employ two versions that only differ in the treatment of zero entries: 
\begin{equation*}
\vari{u} := \sum_{i} \mathbb{1}_{\mathbb{R}_{< 0}}(\tilde{u}_i \tilde{u}_{i+1}), 
    \quad 
    \vari{0} := -1
\end{equation*}
where $\tilde{u}$ is the vector resulting from deleting all zeros in $u$ and $\mathbb{1}_{\mathbb{A}}(x)$ is the indicator function with subset $\mathbb{A}$, i.e., $\mathbb{1}_{\mathbb{A}}(x) = 1$ if $x \in \mathbb{A}$ and zero otherwise. Further, {the \emph{strict variation} is defined by}
\begin{equation*}
\variz{u} := \sum_{i} \mathbb{1}_{\mathbb{R}_{< 0}}(\bar{u}_i \bar{u}_{i+1}), 
\end{equation*}
where $\bar{u}$ is the vector resulting from replacing zeros by elements that maximize the resulting sum. Obviously, $\vari{u} \leq \variz{u}$, {but equality does not necessarily need to hold, e.g., $\vari{\begin{bmatrix}
		1 & 0 & 2
\end{bmatrix}} = \vari{\begin{bmatrix}
		1 & 2
\end{bmatrix}} = 0$, but $\variz{\begin{bmatrix}
		1 & 0 & 2
\end{bmatrix}} = \vari{\begin{bmatrix}
		1 & -1 & 2
\end{bmatrix}}  = 2$. }

Most essential to this work is the definition of variation bounding. 
\begin{defn}
\label{def:svb_k} \label{def:vb_k}
A linear map $u \mapsto X u$ is said to be 
\begin{enumerate}
    \item \emph{(strictly) $k$-variation bounding}, $k \in \mathds{Z}_{\ge0}$ if for all $u\neq 0$ with $\vari{u} \leq k$ it holds that $\vari{Xu} \leq k$ ($\variz{Xu} \leq k$)
    \item  \emph{(strictly) $k$-variation diminishing}, if $X$ is (strictly) $j$-variation bounding for all $j \in (0:k)$. If additionally, the sign of the first non-zero {element} in $u$ and $Xu$, {respectively}, coincide whenever $\vari{u} = \vari{Xu}$ ($\vari{u} = \variz{Xu}$), we say that $X$ is \emph{(strictly) order-preserving $k$-variation diminishing}.
\end{enumerate}
\end{defn}

\subsection{Total Positivity Theory}
\emph{Total positivity theory} \cite{karlin1968total} is known to provide algebraic conditions for variation bounding/diminishing operators property by means of compound matrices. For $X \in \Rmn$, the $(i,j)$-th entry of the so-called \emph{r-th multiplicative compound matrix} 
$\compound{X}{r} \in \mathbb{R}^{\binom{n}{r} \times \binom{m}{r}}$ is defined via $|\submatrix{X}{I}{J}|$, where $I$ and $J$ are the $i$-th and $j$-th element of the $r$-tuples in $\mathcal{I}_{n,r}$ and 
$\mathcal{I}_{m,r}$, respectively.

For example, if $X \in \mathbb{R}^{3 \times 3}$, then
\begin{align*}
 \compound{X}{2} = \begin{bmatrix}
|X_{\{1,2 \},\{1,2 \}}| & |X_{\{1,2 \},\{1,3\}}| & |X_{\{1,2 \},\{2,3\}}|\\
|X_{\{1,3 \},\{1,2 \}}| & |X_{\{1,3 \},\{1,3\}}| & |X_{\{1,3 \},\{2,3\}}|\\
|X_{\{2,3 \},\{1,2 \}}| & |X_{\{2,3 \},\{1,3\}}| & |X_{\{2,3 \},\{2,3\}}|\
\end{bmatrix}
\end{align*}
In case of $k=0$, we define $\compound{X}{0} = 1$. In our derivations, the following properties of the multiplicative compound matrix will be elementary (see, e.g., \cite[Section~6]{fiedler2008special} and \cite[Subsection~0.8.1]{horn2012matrix}).
\begin{lem}\label{lem:compound_mat}
	Let $X \in \mathbb{R}^{n \times p}$ and $Y \in \mathbb{R}^{p \times m}$.
	\begin{enumerate}[i)]
		\item $\compound{(XY)}{r} = \compound{X}{r}\compound{Y}{r}$ (Cauchy-Binet formula). \label{item:Cauchy_Binet}
		\item  For $ p = n $: $\sigma(\compound{X}{r}) = \{\prod_{i \in I} \lambda_i(X): I \in \mathcal{I}_{n,r} \}$. \label{item:compound_eigen}
		
	\end{enumerate} 
\end{lem}
It is readily seen that $X$ is order-preserving $0$-variation diminishing if and only if $X = \compound{X}{1} \geq 0$. This equivalence can be generalized to higher orders by the following definitions and characterizations (see~\cite[Prop.~7]{grussler2020variation} and \cite[Theorem 5.1.1]{karlin1968total}). 
\begin{defn}\label{def:k_pos_matrix}
    Let $X \in \Rmn$ and $k \leq \min\{m,n\}$. $X$ is called 
    \begin{enumerate}[i.]
        \item \emph{(strictly) $k$-sign consistent} if $\varepsilon\bigl(\compound{X}{k}\bigr) \geq (>) 0$ for some $\varepsilon \in \{-1,1\}$.
        \item \emph{(strictly) $k$-sign regular} if $X$ is (strictly) $j$-sign consistent for all $j \in (1:k)$. 
        \item \emph{(strictly) $k$-positive} if $\compound{X}{j} \geq (>) 0$ for all $j \in (1:k)$. In case of $k = \min\{m,n\}$, $X$ is also called \emph{(strictly) totally positive}.
        \end{enumerate}

\end{defn}
\begin{prop}\label{prop:sc_k_mat_vb_m}
	For $X \in \Rmn$, $n \geq m$, the following hold:
	\begin{enumerate}
		\item if $k \leq \min\{m,n\}-1$, then $X$ is strictly $k-1$-variation bounding if and only if $X$ is strictly $k$-sign consistent. 
		\item if $\rk(X) = m < n$, then $X$ is $m-1$-variation bounding if and only if $X$ is $m$-sign consistent. \label{item:SC_full_rank}\label{item:SC_low_rank}
		\item if $k < \rk(X)$ and any $k$ columns of $X$ are linearly independent, then $X$ is $k-1$-variation bounding if and only if $X$ is $k$-sign consistent  \label{item:SC_lin_ind}
	\end{enumerate}
\end{prop}

Next, we will review how checking the sign of the elements in $\compound{X}{k}$ can be simplified compared to computing all of its entries. 

 A related result are the following conditions for (strict) sign consistency (see~\cite[Theorem 2.2]{pena_matrices_1995} \& \cite[Theorem~25]{grussler2024system}). 
\begin{prop}\label{prop:k_ssc__mat_pena_i}
    Let $X \in \Rmn$ be such that $n \geq 2m$ and all $m$-minors $|\submatrix{X}{\alpha}{\nset{m}}|$ defined by
    \begin{enumerate}[i.]
        \item $\alpha = \{\nset{m-r},(t:t+r-1)\}$ with $1\leq r < m $ and $m-r < t \leq n-r+1$, and $\alpha = (t:t+m-1),\;\; 1 \leq t \leq m$ have the same strict sign $\varepsilon \in \{-1,1\}$.
        \item $\alpha = (t:t+m-1),\;\; m+1 \leq t \leq n-m+1$ have the same (strict) sign in $\{0,\varepsilon\}$. 
    \end{enumerate}
Then, $X$ is (strictly) $m$-sign consistent. In the strict case, this is also a necessary condition.
\end{prop}
\cref{prop:k_ssc__mat_pena_i} can also be applied to matrices with dimensions $m > n$ by consideration of $X^\transp$. Checking strict $2$-sign consistency of
\begin{equation*}
    X = \begin{bmatrix}
        1 & 1 & 1 & 1 \\ 
        1 & 2 & 3 & 4 
    \end{bmatrix} \in \mathds{R}^{2 \times 4}
\end{equation*}
only requires verifying that
\begin{align*}
    r=1:& \quad \det{\begin{bmatrix} 1 & 1\\ 1 & 2 \end{bmatrix}},\;
\det{\begin{bmatrix} 1 & 1 \\ 1 & 3 \end{bmatrix}}, \;
\det{\begin{bmatrix} 1 & 1 \\ 1 & 4 \end{bmatrix}}
    \\
    r=2:& \quad 
    \det{\begin{bmatrix} 1 & 1\\ 1 & 2 \end{bmatrix}}, \;
    \det{\begin{bmatrix} 1 & 2 \\ 1 & 3 \end{bmatrix}}, \;
    \det{\begin{bmatrix} 1 & 3\\ 1 & 4 \end{bmatrix}}.
\end{align*}
share the same strict sign. In case that the first $m-1$ columns of $X \in \mathds{R}^{n \times m}$, $n \geq m$, are strictly $m-1$-sign consistent, (strict) $m$-sign consistency of $X$ only requires verification of the consecutive $m$-minors \cite[Theorem~3.3.2]{karlin1968total}.
\begin{prop}\label{prop:m_1_to_m}
    Let $X \in \mathds{R}^{n \times m}$, $n \geq m$, be such that 
    \begin{enumerate}[i.]
        \item $X_{(1:n),(1:m-1)}$ is strictly $m-1$-sign consistent with sign $\varepsilon \in \{-1,1\}$. 
        \item all consecutive $m$-minors share the same (strict) sign in $\{0,\varepsilon\}$.
    \end{enumerate}
   Then, $X$ is (strictly) $m$-sign consistent. 
\end{prop}

\subsection{System Operators}
We consider finite-dimensional DTLTI systems  \((A,b,c)\) of the form
\begin{equation}\label{eq:SISO_d}
\begin{aligned}
    x(t+1) &= A x(t) + b\,u(t),\\
    y(t)   &= c x(t),
\end{aligned}
\end{equation}
where \(A\in\mathbb{R}^{n\times n}\) and \(b,c^\top\in\mathbb{R}^n\). Its impulse response, i.e., the output $y(t)$ to the input $u(t) = \mathbf{1}_{\{0\}}$, computes as 
\begin{equation*}
    g(t) = \begin{cases}
        cA^{t-1}b & t \geq 1\\
        0 & t \leq 0
    \end{cases}
\end{equation*}
and we say that the system is \emph{(strictly) externally positive} (\ref{eq:SISO_d}) if $g(t) \geq (>) 0$ for all $t \geq 1$. The following system operators associated with \cref{eq:SISO_d} %
will be used throughout this work:
\begin{enumerate}
    \item \emph{(finite-time) controllability \&
	observability operators}
    	\begin{align*}
		\Con{N}(A,b) &:= \begin{bmatrix}
			b & Ab & \dots & A^{N-1}b
		\end{bmatrix} \\
		\Obs{N}(A,c) &:= \begin{bmatrix}
			c^\transp & A^\transp  c^\transp & \dots & {A^\transp}^{N-1} c^\transp
		\end{bmatrix}^\transp
	\end{align*}
    where for $N=0$, they are defined to be empty matrices. 
    \item \emph{Hankel operator}
  \begin{align*}
	(\Hank{g} u)(t) &:= \sum_{\tau=-\infty}^{-1} g(t-\tau)u(\tau)
	= \sum_{\tau=1}^{\infty} g(t+\tau)u(-\tau).
	\label{eq:def_hank_disc}  
\end{align*} \normalsize
\end{enumerate}
$\Hank{g}$ is the limit (for $N\to\infty$) of the finite truncated matrix representations $\Hank{g}^N u = H_g(1,N,N)u(-1:-N)$, 
where we define \small \begin{align*}%
H_g(t,i,j) &:= \begin{bmatrix}
	g(t) & g(t+1) & \dots & g(t+j-1) \\
	g(t+1) & g(t+2) & \dots & g(t+j) \\
	\vdots & \vdots & & \vdots \\
	g(t+i-1) & g(t+i) & \dots & g(t+i+j-2)
\end{bmatrix}\\ &= \Obs{i}(A,c) A^{t-1} \Con{j}(A,b), \; t \geq 1
\end{align*}\normalsize 
Note that analogously to the discussion in \cite[Lemma~3.4]{grussler2021internally}, $\Hank{g}$ is (strictly) $k$-sign consistent if and only if $H_g(1,N,N)$ is (strictly) $k$-sign consistent for all $N \geq k$. In particular, by \cite{grussler2020variation} the following hold. 
\begin{prop}\label{prop:Hankel_kpos}
    Let $(A,b,c)$ be as in \cref{eq:SISO_d} have impulse response $g$. Then, the following are equivalent: 
\begin{enumerate}
    \item $\Hank{g}$ is (strictly) order-preserving $(k-1)$-variation diminishing. 
    \item $\Hank{g}$ is (strictly) $k$-positive.
    \item $H_g(1,k,N)$ is (strictly) $k$-positive for all $N \geq k$.
\item The \emph{consecutive compound systems} 
\begin{equation}
    \label{eq:comp_real}
    (\compound{A}{j}, \compound{\Con{j}(A,b)}{j},\compound{\Obs{j}(A,c)}{j})
\end{equation}
are (strictly) externally positive for $j \in (1:k)$. 
\end{enumerate}
\end{prop}
A main advantage of the compound system characterization in \cref{prop:Hankel_kpos} lies in its tractability via analytic and computational certificates such as \cite{grussler2019tractable,taghavian2023external,farina2011positive,drummond2023externally,weller2023external}. Similar characterizations have also been shown for the observability \& controllability operators in the $k$-positive \cite{grussler2021internally} and $k$-sign consistent cases \cite{grussler2024system}. 

\section{Main Results}
\label{sec:main}
The objective of this work is to make the identification of (strict) $k$-variation bounding Hankel operators $\Hank{g}$ tractable and analyzable via the derivation of generalized compound systems, analogously to \cref{prop:Hankel_kpos}. To this end, we recall the following recently derived partial analogue \cite{grussler2026efficient}.
\begin{prop}\label{prop:reduced_hankel_SC_k}
     Let $(A,b,c)$ be as in \cref{eq:SISO_d} have impulse response $g$. Then, the following are equivalent: 
     	\begin{enumerate}
		\item $\Hank{g}$ is (strictly) $k$-sign consistent. 
		\item $H_g(1,k,N)$ is (strictly) $k$-sign consistent for all $N \geq k$. 
	\end{enumerate}
\end{prop}
By \cref{prop:reduced_hankel_SC_k}, it suffices to apply \cref{prop:k_ssc__mat_pena_i} to $H_g(1,k,N)$ for all $N \geq k$ in order to check that $\Hank{g}$ is (strictly) $k$-sign consistent. Equivalently, this requires checking the (strict) sign consistency among $k$-minors of the form
\begin{multline}
    \begin{vmatrix}
        H_g(1,k,k-r) & H_g(t,k,r)
    \end{vmatrix} = \\
    \compound{\Obs{k}(A,c)}{k} \begin{bmatrix}
        \Con{k-r}(A,b) & A^{t-1} \Con{r}(A,b)
    \end{bmatrix}_{[k]}, \label{eq:Hankel_k_minors}
\end{multline}
for all $1 \leq r \leq k$ and $t \geq k-r+1$, where the equality is a result of the Cauchy-Binet formula (see~\cref{lem:compound_mat}). Due to the multi-linearity of the compound matrix, inherited from the standard determinant, and the Cauchy-Binet formula, there must exist an $F^{\Con{k-r}(A,b),k}\in \mathds{R}^{\binom{n}{k} \times \binom{n}{r}}$ such that 
\begin{multline}
    \begin{bmatrix}
        \Con{k-r}(A,b) & A^{t-1} \Con{r}(A,b)
    \end{bmatrix}_{[k]} = \\ F^{\Con{k-r}(A,b),k} \compound{A}{r}^{t-1} \compound{\Con{r}(A,b)}{r}. \label{eq:contr_op_compound}
\end{multline}
Thus, the $k$-minors in \cref{eq:Hankel_k_minors} correspond to the impulse response of systems with realization 
\begin{equation*}
    (\compound{A}{r},\compound{A}{r}^{k-r}\compound{\Con{r}(A,b)}{r}, \compound{\Obs{k}(A,c)}{k} F^{\Con{k-r}(A,b),k})
\end{equation*}
 In order to provide an explicit formula for $F^{\Con{k-r}(A,b),k}$, we derive the following generic lemma.
\begin{lem}\label{lem:compound_multi_lin}
    Let $X = \begin{bmatrix}
        Y & Z
    \end{bmatrix}\in \Rmn$, $1 \leq r \leq m \leq n$, $Y \in \mathds{R}^{n \times m-r}$ and $Z \in \mathds{R}^{n \times r}$. 
    Then, $$\compound{X}{m} = F^{Y,m} \compound{Z}{r}$$,
    
    where $F^{Y,m} \in \mathds{R}^{\binom{n}{m} \times \binom{n}{r}}$ is defined for $r < m$ by
    \begin{equation*}
    F^{Y,m}_{i,j} := \begin{cases}
        (-1)^{\varepsilon(I\setminus J,I)}   |Y_{I\setminus J,(1:m-r)}| & \text{if } J \subset I\\ 
        0 & \text{if } J \not \subset I
    \end{cases}
    \end{equation*}
    with $\varepsilon(I\setminus J,I) := \sum_{i \in I \setminus J} \pos{I}{i} + \frac{(m-r)(m-r+1)}{2}$, and $I$ and $J$ are the $i$-th and $j$-th element of $\mathcal{I}_{n,m}$ and 
$\mathcal{I}_{n,r}$, respectively. In case that $r = m$, $F^{Y,m} = I_{\binom{n}{m}}$. 

\end{lem}
\begin{proof}
Let $I \in \mathcal{I}_{n,m}$ be the $i$-th element in $\mathcal{I}_{n,m}$. Then, by the definition of the compound matrix, ${\compound{X}{m}}_{i} = |X_{I,(1:m)}|$. Using the generalized Laplace expansion \cref{eq:gen_laplace} along the first $m-r$ columns, it follows from our definition of $Y$ and $Z$ that 
\begin{equation*}
    {\compound{X}{m}}_{i} = \sum_{\substack{S \subset I\\ S \in\mathcal{I}_{n,m-r}}}(-1)^{\varepsilon(S,I)}|Y_{S,(1:m-r)}||Z_{I\setminus S,(1:r)}|,
\end{equation*}
where 
\begin{align*}
\varepsilon(S,I) &:= \sum_{s \in S}pos_I(s)+ \sum_{1 \leq i \leq m-r}i\\ 
&= \sum_{s \in S}pos_I(s) + \frac{(m-r)(m-r+1)}{2}. 
\end{align*}
 Therefore, by re-indexing the sum via the complements of all $S \in\mathcal{I}_{n,m-r}$, $J = I\setminus S \in \mathcal{I}_{n,r}$, we get that
\begin{align*}
 {\compound{X}{m}}_{i} &= \sum_{\substack{J \subset I\\ J \in\mathcal{I}_{n,r}}}(-1)^{\varepsilon(I\setminus J,I)}\small|Y_{I\setminus J,(1:m-r)}| |Z_{J,(1:r)}| \\
 &= \sum_{j = 1}^{\binom{n}{r}} F^{Y,m}_{i,j} {\compound{Z}{r}}_j,  
\end{align*}
which proves our formula. 

\end{proof}

In the following, let us illustrate \cref{lem:compound_multi_lin} by an example. 
\begin{ex}
Let $b \in \mathds{R}^3$, $A \in \mathds{R}^{3 \times 3}$ and $X = \begin{bmatrix}
    b & A^tb
\end{bmatrix}$. Then, by \cref{lem:compound_multi_lin} $\compound{X}{2} = F^{b,2} A^t b$,
where
\begin{align*}
  &\begin{matrix}
        J = \{1\} & J = \{2\} & J = \{3\}
    \end{matrix} \\
   F^{b,2} = &   \begin{bmatrix}
        -b_2 &\qquad b_1&\qquad0\\-b_3&\qquad 0&\qquad b_1\\0&\qquad -b_3&\qquad b_2
    \end{bmatrix}    \begin{matrix}
        I=\{1,2\} \\I=\{1,3\} \\I=\{2,3\}
    \end{matrix}
\end{align*}

\end{ex}

Consequently, application of \cref{prop:sc_k_mat_vb_m}, yields the following main result: a complete characterization of strict $k$-variation bounding Hankel operators in terms of \emph{generalized compound systems}. 
\begin{thm}
\label{thm:ssck_via_compound_sign}
Let $(A,b,c)$ in \cref{eq:SISO_d} have impulse response $g$ and $1 \leq k \leq n$. Then, the following are equivalent: 
\begin{enumerate}
\item $\Hank{g}$ is strictly $k-1$-variation bounding. 
    \item $\mathcal{H}_g$ is strictly $k$-sign consistent. 
    \item All \emph{generalized compound systems} 
\begin{equation}
 {\left(\compound{A}{r},\compound{A}{r}^{k-r}\compound{\Con{r}(A,b)}{r},\varepsilon \compound{\Obs{k}(A,c)}{k} F^{\Con{k-r}(A,b),k}\right)}, \label{eq:gen_compound_sys}
\end{equation}
 are strictly externally positive, $r \in (1:k)$, where $\varepsilon := \sign(|H_g(1,k,k)|)$.
\end{enumerate}
\end{thm}

Unlike the consecutive compound systems in \cref{eq:comp_real}, which encode consecutive Hankel minors, \cref{eq:gen_compound_sys} also represents minors formed by a fixed initial column block and a moving consecutive column block. Similarly, using the non-strict part of \cref{prop:k_ssc__mat_pena_i}, we arrive at the following non-strict extension.
\begin{cor}
    Let $(A,b,c)$ in \cref{eq:SISO_d} have impulse response $g$ and $1 \leq k \leq n$. Further, assume that
    \begin{enumerate}
        \item the generalized compound systems in \cref{eq:gen_compound_sys} are strictly externally positive for $r \in (1:k-1)$
        \item $(\compound{A}{k}, \compound{\Con{k}(A,b)}{k},\varepsilon \compound{\Obs{k}(A,c)}{k})$ is externally positive and $\varepsilon \compound{\Obs{k}(A,c)}{k} \compound{A}{k}^{t-1} \compound{\Con{k}(A,b)}{k} > 0$ for all $t \in (1:2k-1)$, where $\varepsilon := \sign(|H_g(1,k,k)|)$. 
    \end{enumerate}
Then, $\Hank{g}$ is $k$-sign consistent and, thus, $k-1$-variation bounding.
\end{cor}
In the special case that $\Hank{g}$ strictly $k-1$-variation bounding, one can utilize \cref{prop:m_1_to_m} to simplify the verification of $k$-variation bounding via the consecutive compound systems.
\begin{cor}\label{cor:Hankel_k_to_k_plus_1}
    Let $(A,b,c)$ in \cref{eq:SISO_d} have impulse response $g$ and $1 \leq k < n$. Further, assume that 
    \begin{enumerate}
        \item the generalized compound systems in \cref{eq:gen_compound_sys} are strictly externally positive for $r \in (1:k)$
        \item there exists an $\varepsilon \in \{-1,1\}$ such that the consecutive compound system $(\compound{A}{k+1}, \compound{\Con{k+1}(A,b)}{k+1}, \varepsilon \compound{\Obs{k+1}(A,c)}{k+1})$ is (strictly) externally positive.
    \end{enumerate}
Then, $\Hank{g}$ is (strictly) $k+1$-sign consistent and, thus, (strictly) $k$-variation bounding.
\end{cor}

\subsection{Toeplitz Operator}
\cref{thm:ssck_via_compound_sign} can also be applied to check whether the \emph{Toeplitz operator}
$$(\Toep{g} u)(t) := \sum_{\tau=0}^{t} g(t-\tau)u(\tau), \ t \geq 0$$
of \cref{eq:SISO_d} is $k-1$-variation bounding. Using the \emph{Toeplitz matrix} \begin{multline*} \scriptsize \hspace{-0.4 cm}
T_g(t,m,n) := \begin{bmatrix}
	g(t) & g(t-1) & \dots & g(t-n+1)\\
	g(t+1) & g(t) & \dots & g(t-n+2)\\
	\vdots & \vdots &  & \vdots \\
	g(t+m-1)  & g(t+m-2) & \dots & g(t+m-n)\\
	\end{bmatrix}
\end{multline*}
it follows from \cite{grussler2026efficient}, that one can equivalently verify that $T_g(N-k,k,2N-k) \in \mathds{R}^{k \times 2N-k}$ is $k$-sign consistent for all \(N \ge k\). In particular, reversing the column order of a Toeplitz matrix results in a
Hankel matrix to which our previous analysis can be applied to. The main differences between the Hankel and Toeplitz case are the following:
\begin{enumerate}[i.)]
    \item For causal systems, $T_g(N-k,k,2N-k)$ will contain zero $k$-minors, which prevents strict $k$-sign consistency.
    \item If $g(t_0) \neq 0$ and $g(t) = 0$ for all $t < t_0$, then $|T_g(t_0,k,k)| = g(t_0)^k$ determines the sign. Even $k$ can only give a positive sign. 
\end{enumerate}
\subsection{Observability Operator}
Variation bounding observability operators
\begin{align*}
 	(\Obs{}(A,c) x_0)(t) &:= cA^t x_0, \ x_0 \in \mathds{R}^n,  \ t \geq 0. 
 \end{align*}
 have been characterized and analyzed in \cite{grussler2024system}. Our presented approach based on \cref{lem:compound_multi_lin} provides a unifying framework that enables determining realizations of the necessary generalized compound systems. For example, using $\mathcal{O}^N(A,c) = \mathcal{C}^N(A^\transp,c^\transp)^\transp$, it follows by \cref{eq:contr_op_compound} that strict $n$-sign consistency of $\Obs{}(A,c)$ is equivalent to the strict external positivity of
\begin{equation*}
\left(A^\transp_{[r]}, {\compound{A^\transp}{r}}^{n-r}\compound{\Con{r}(A^\transp,c^\transp)}{r},|\Obs{n}(A,c)|F^{\Con{n-r}(A^\transp,c^\transp),n}\right).
\end{equation*}
for all $r \in (1:n)$. 
\subsection{Pole Configuration}
As seen in \cite{grussler2020variation,grussler2024system}, variation bounding properties of system operators induce restrictions on the system's pole configuration. This also remains true for strictly sign consistent Hankel operators as shown here for $k=2$. 
\begin{cor}\label{cor:poles}
Let $(A,b,c)$ be a minimal realization with impulse response $g$, (strictly) $1$-variation bounding $\Hank{g}$ and $\lambda_1(A) \neq \lambda_2(A) \neq \lambda_3(A)$. Then, $\lambda_1(A) > \lambda_2(A) \geq (>)0$ 
\end{cor}

\begin{proof}
Since $\Hank{g}$ is $1$-variation bounding, it follows for the $0$-variation input $u = \mathbf{1}_{\{0\}}$ with $g(t) = \Hank{g}u(t)$, $t \geq 1$, that $\vari{g} \leq 1$. Hence, there exists a $t_0 \in \mathds{Z}_{\geq 0}$ such that $(A,A^{t_0}b,\sign(cA^{t_0}b)c)$ is externally positive, which requires that $\lambda_1(A) \in \mathds{R}_{\geq 0}$ (see, e.g., \cite{farina2011positive}). In particular, if $\lambda_1(A) = 0$, then the system is of finite impulse response, which would contradict that $\variz{\Hank{g}u} \leq 1$ when $\Hank{g}$ is strictly $1$-variation bounding.

Finally, by \cref{thm:ssck_via_compound_sign}, the consecutive compound system 
\begin{equation}
    (\compound{A}{2}, \compound{\Con{2}(A,b)}{2}, \sign(|H_g(1,2,2)|)\compound{\Obs{2}(A,c)}{2}) \label{eq:second_comp}
\end{equation} is required to be (strictly) externally positive. Note that while \cref{thm:ssck_via_compound_sign} is stated only for the strict case, in the non-strict case, external positivity remains necessary because the impulse response of \cref{eq:second_comp} corresponds to signed Hankel minors that must be nonnegative. under the assumption that $\lambda_1(A) \neq \lambda_2(A) \neq \lambda_3(A)$, it can be shown as in \cite[Theorem~15]{grussler2020variation} that $\lambda_1(A)\lambda_2(A)$ is an observable and controllable mode of \cref{eq:second_comp}. Analogously to above, this necessitates that $\lambda_1(A)\lambda_2(A) \geq (>) 0$, which concludes the proof. 
\end{proof}

The pole restriction in \cref{cor:poles} differs from that of autonomous variation-bounding systems $$x(t+1)=A x(t),$$ where $A$ is required to be $k$-sign consistent \cite{alseidi2021discrete}. To see this, consider the strictly $2$-sign consistent
\begin{equation*}
A =
\begin{bmatrix}
7 & -1 & -5\\
5 & 7 & -1\\
1 & 5 & 7
\end{bmatrix}\quad \text{with} \quad
\compound{A}{2}= 18
\begin{bmatrix}
3 & 1 & 2\\
2 & 3 & 1\\
1 & 2 & 3
\end{bmatrix}>0,
\end{equation*}
with 
\begin{equation*}
\sigma(A)
=
\left\{
9+3\sqrt{3}\,\mathrm{i},
9-3\sqrt{3}\,\mathrm{i},
3
\right\},
\end{equation*}
The system is strictly $1$-variation bounding:  $\vari{x(0)} \leq 1$ implies $\vari{x(t)} \leq 1$ for all $t \geq 0$, but its two dominant eigenvalues are non-real. Thus, bounding the variation across state components does not impose the same dominant-pole restriction as variation bounding through the Hankel operator. 

\section{EXAMPLE}\label{sec:ex}
In the following, we illustrate our results based on a discretized
\emph{three-node lumped-parameter thermal system}, which also demonstrates how $1$-variation bounding Hankel operators arise in physical systems.

Consider the three-node thermal chain in \cref{fig:three_node_interconnection}, extending the lumped-capacity model in \cite[p.~201]{lienhard2024heat}. The states $\tilde T_i=T_i-T_{air}$ are temperatures relative to ambient, $u(t)$ is the heat input rate to node~1, and $y=\tilde T_3-\gamma\tilde T_1$. Each node has thermal capacitance $c_i$. The resistances $R_i$ govern heat losses $\dot Q_i$ to ambient, while $R_{12},R_{23}$ govern inter-node heat flows $\dot Q_{12},\dot Q_{23}$.
\begin{figure}[t]
\centering
\begin{tikzpicture}[
    scale=1,
    transform shape,
    >=Latex,
    thick,
    font=\normalsize,
    block/.style={
        draw,
        minimum width=1cm,
        minimum height=1cm,
        align=center
    },
    gain/.style={
        draw,
        minimum width=0.7cm,
        minimum height=0.45cm,
        align=center
    },
    line/.style={-Latex, thick}
]

\node[block] (b1) {$\tilde T_1$};
\node[block, right=1.35cm of b1] (b2) {$\tilde T_2$};
\node[block, right=1.35cm of b2] (b3) {$\tilde T_3$};

\draw[line] ($(b1.west)+(-0.8,0)$) -- (b1.west)
    node[midway, above] {$u$};

\draw[line] (b1.north) -- ++(0,0.75)
    node[midway, right] {$\dot Q_1$};

\draw[line] (b2.north) -- ++(0,0.75)
    node[midway, right] {$\dot Q_2$};

\draw[line] (b3.north) -- ++(0,0.75)
    node[midway, right] {$\dot Q_3$};

\draw[line] ($(b1.east)+(0,0.13)$) -- ($(b2.west)+(0,0.13)$)
    node[midway, above] {$\dot Q_{12}$};

\draw[line] ($(b2.west)+(0,-0.13)$) -- ($(b1.east)+(0,-0.13)$);

\draw[line] ($(b2.east)+(0,0.13)$) -- ($(b3.west)+(0,0.13)$)
    node[midway, above] {$\dot Q_{23}$};

\draw[line] ($(b3.west)+(0,-0.13)$) -- ($(b2.east)+(0,-0.13)$);

\node[draw, circle, minimum size=0.55cm, right=0.7cm of b3] (sum) {$-$};

\draw[line] (b3.east) -- (sum.west);

\draw[line] (sum.east) -- ++(0.8,0)
    node[midway, above] {$y$};

\coordinate (fbdown) at ($(b1.south)+(0,-0.45)$);
\coordinate (gleft)  at ($(sum.south)+(-4.15,-0.6)$);
\coordinate (gright)  at ($(sum.south)+(-3.45,-0.6)$);
\coordinate (sumsouth) at ($(sum.south)+(0,-0.6)$);

\draw[thick] (b1.south) -- (fbdown);
\draw[line] (fbdown) -- (gleft);

\node[gain] (gainblock) at ($(sum.south)+(-3.8,-0.6)$) {$\gamma$};

\draw[thick] (gright) -- (sumsouth);

\draw[line] (sumsouth) -- (sum.south);

\end{tikzpicture}
\caption{Three-node lumped-parameter thermal interconnection.}
\label{fig:three_node_interconnection}
\end{figure}
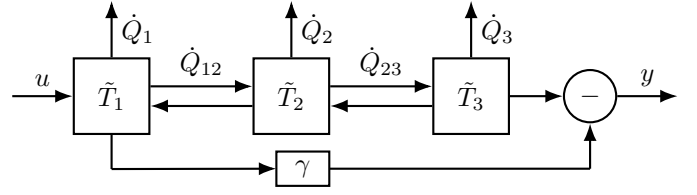

The node-wise energy balances are
\begin{align*}
c_1 \dot{\tilde{T}}_1 &= u(t) - \frac{\tilde T_1}{R_1} - \frac{\tilde T_1-\tilde T_2}{R_{12}} \\
c_2 \dot{\tilde{T}}_2 &= - \frac{\tilde T_2}{R_2} + \frac{\tilde T_1-\tilde T_2}{R_{12}} - \frac{\tilde T_2-\tilde T_3}{R_{23}} \\
c_3 \dot{\tilde{T}}_3 &= - \frac{\tilde T_3}{R_3} + \frac{\tilde T_2-\tilde T_3}{R_{23}}
\end{align*}
For $c_1=1$, $c_2=1.5$, $c_3=2$, $R_{12}=1$, $R_{23}=1.5$, $R_1=R_2=R_3=5$, and $\gamma=0.2$, zero-order-hold discretization with sampling interval $\Delta t=0.5$ gives the system in \cref{eq:SISO_d} with
\begin{equation}
\begin{aligned}
A &= \begin{bmatrix} 
0.5950 & 0.2809 & 0.0351 \\
0.1872 & 0.5942 & 0.1517 \\
0.0176 & 0.1138 & 0.8185
\end{bmatrix}, \quad
b = \begin{bmatrix} 0.3849 \\ 0.0569 \\ 0.0033 \end{bmatrix}\\
c &= \begin{bmatrix} -0.2 & 0 & 1 \end{bmatrix}.
\end{aligned} \label{eq:three_node_lumped}
\end{equation}
Intuitively, for the impulse input $u(t) = \mathbf{1}_{\{0\}}$, the output will be initially dominated by a rapid rise of $x_1$, which causes a negative excursion of the output due to $-\gamma$. As heat diffuses from node~2 to node~3, $x_3$ eventually prevails, pulling $y(t)$ towards positive values. Therefore, as illustrated in \cref{fig:thermal_responses_3node}, the impulse response $g$ satisfies $\variz{g} = 1$. Next, let us verify that $\Hank{g}$ is strictly $1$-variation bounding. By \cref{thm:ssck_via_compound_sign}, we have to check that the generalized compound systems
\begin{align*}
    r = 1: & \left(A,Ab,\varepsilon \compound{\Obs{2}(A,c)}{2} F^{b,2}\right)\\
    r = 2: & \left(A_{[2]}, \compound{\Con{2}(A,b)}{2}, \varepsilon \compound{\Obs{2}(A,c)}{2}\right)
\end{align*}
with $\varepsilon = \sign|H_g(1,2,2)|$, are strictly externally positive.

This is illustrated graphically in \cref{fig:thermal_responses_3node}, but can also be achieved numerically using, e.g., \cite{taghavian2023external} or the discrete-time version of \cite{grussler2019tractable} as shown in \cite{grussler2017identification}. Moreover, note that the system poles are (up to rounding errors) given by $\sigma(A) = \{0.9366, 0.7187, 0.3523\} \subset \mathds{R}_{> 0}$, confirming \cref{cor:poles}.
\begin{figure}[htbp]
    \centering
    \begin{tikzpicture}
        \begin{axis}[
            name=plot1,
            width=0.9\columnwidth,
            height=0.35\columnwidth,
            ylabel={$g(t)$},
            grid=major,
            xmin=0, xmax=50,
            axis x line=bottom,
            axis y line=left,
        ]
        \addplot [
            ycomb, color=blue, thick, mark=*, mark options={fill=blue, scale=0.6}
        ] table {impulse_response.txt}; 
        
        \end{axis}

        \begin{axis}[
            name=plot2,
            at={(plot1.south west)},
            anchor=north west,      
            yshift=-0.7cm,        
            width=0.9\columnwidth, 
            height=0.35\columnwidth, 
            ylabel={$g_{1,2}(t)$},
            grid=major,
            xmin=0, xmax=50,
            axis x line=bottom,
            axis y line=left,
        ]
        \addplot [
            ycomb, color=red, thick, mark=*, mark options={fill=red, scale=0.6}
        ] table {sigma_km_response.txt};
        
        \end{axis}

        \begin{axis}[
            name=plot3,
            at={(plot2.south west)},
            anchor=north west,      
            yshift=-0.7cm,      
            width=0.9\columnwidth, 
            height=0.35\columnwidth, 
            xlabel={$t$},
            ylabel={$g_{2,2}(t)$},
            grid=major,
            xmin=0, xmax=50,
            axis x line=bottom,
            axis y line=left,
        ]
        \addplot [
            ycomb, color=purple, thick, mark=*, mark options={fill=purple, scale=0.6}
        ] table {sigma_vals_consecutive.txt};
        
        \end{axis}
        
    \end{tikzpicture}
    
    \caption{Impulse responses for the three-node thermal system with matrices in \cref{eq:three_node_lumped}: (top) original system; (middle) generalized compound system with $r=1$, $k=2$; (bottom) consecutive compound system with $r=2$, $k=2$. The strict external positivity of the compound systems implies that $\Hank{g}$ is strictly $1$-variation bounding. Since $\vari{g}>0$, $\Hank{g}$ is not $0$-variation bounding and hence is not $2$-variation diminishing.}
    \label{fig:thermal_responses_3node}
\end{figure}
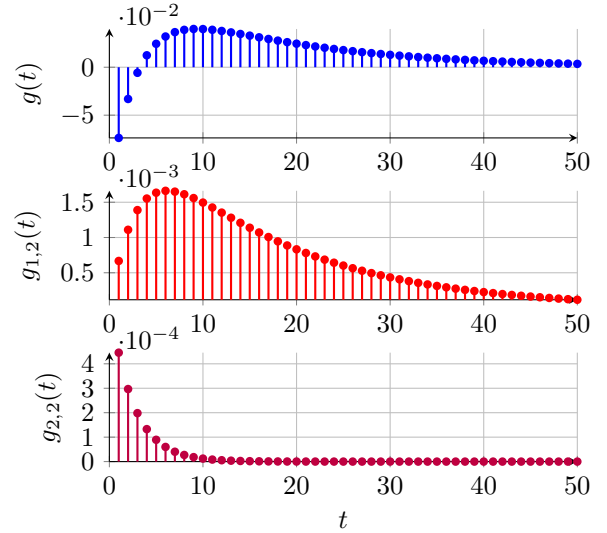

Finally, since $\compound{A}{3}, \compound{\Con{3}(A,b)}{3}, \compound{\Obs{3}(A,c)}{3} \in \mathds{R} \setminus \{0\}$, it follows from \cref{cor:Hankel_k_to_k_plus_1} that $\Hank{g}$ is also strictly $2$-variation bounding. Notably, while the system does not belong to the conventional external and internally positive class, it still exhibits many generalized positivity properties, which so far have not been explored and utilized. 

\section{CONCLUSION}
\label{sec:conc}
In this work, we have derived a certificate for when the Hankel operator of a DTLTI system leaves the set of signals with at most $k-1$ sign changes invariant. Our certificate parallels earlier works \cite{grussler2020variation,grussler2024system} by encoding this property as the (strict) external positivity of $k$ generalized DTLTI compound systems. The latter can be efficiently checked using numerical methods such as \cite{taghavian2023external,grussler2019tractable,drummond2023externally}.
In particular, while naive verification of $k$-sign-consistency would require checking infinitely many Hankel minors in both row and column directions, our result reduces this to external-positivity tests of only $k$ finite-dimensional generalized compound systems. Our new approach extends and unifies earlier investigations on $k$-variation bounding observability operators \cite{grussler2024system}, as well as $k$-variation diminishing Hankel and Toeplitz operators \cite{grussler2020variation,grussler2021internally}.
Furthermore, we show that the two largest dominant poles of a (strictly) $1$-variation bounding Hankel operator need to be real and (strictly) positive. In future work, we will generalize this result to general $k$, derive decompositions of the system's transfer function (cf.~\cite{grussler2020variation}), and analyze its impacts on model order reduction (cf.~\cite{grussler2020balanced}).
\printbibliography

\end{document}